\documentclass[12pt,reqno]{article}

\usepackage[usenames]{color}
\usepackage{amssymb}
\usepackage{amsmath}
\usepackage{amsthm}
\usepackage{amsfonts}
\usepackage{amscd}
\usepackage{graphicx}

\usepackage[colorlinks=true,
linkcolor=webgreen,
filecolor=webbrown,
citecolor=webgreen]{hyperref}

\definecolor{webgreen}{rgb}{0,.5,0}
\definecolor{webbrown}{rgb}{.6,0,0}

\usepackage{color}
\usepackage{fullpage}
\usepackage{float}

\usepackage{graphics}
\usepackage{latexsym}
\usepackage{epsf}
\usepackage{breakurl}

\def\suchthat{\, : \,}
\def\modd#1 #2{#1\ \mbox{\rm (mod}\ #2\mbox{\rm )}}
\DeclareMathOperator{\Pal}{Pal}
\def\Palx{\Pal_{\bf x}}
\DeclareMathOperator{\Fac}{Fac}
\def\rhox{\rho_{\bf x}}

\begin{document}

\title{Palindrome complexity versus factor complexity}

\author{Jeffrey Shallit\\
School of Computer Science\\
University of Waterloo\\
Waterloo, ON  N2L 3G1 \\
Canada\\
\href{mailto:shallit@uwaterloo.ca}{\tt shallit@uwaterloo.ca}}

\maketitle

\theoremstyle{plain}
\newtheorem{theorem}{Theorem}
\newtheorem{corollary}[theorem]{Corollary}
\newtheorem{lemma}[theorem]{Lemma}
\newtheorem{proposition}[theorem]{Proposition}
\newtheorem{claim}{Claim}

\theoremstyle{definition}
\newtheorem{definition}[theorem]{Definition}
\newtheorem{example}[theorem]{Example}
\newtheorem{conjecture}[theorem]{Conjecture}

\theoremstyle{remark}
\newtheorem{remark}[theorem]{Remark}

\def\Enn{\mathbb{N}}

\begin{abstract}
Let ${\bf x} = (a_i)_{i \geq 0}$ be an infinite word over a finite alphabet
$\Sigma$.  Let $\rho (n)$ be the factor complexity function for $\bf x$ and
$\Pal(n)$ be the palindrome complexity function for $\bf x$.
We give a new relationship between these two quantities; namely,
if $\bf x$ is not ultimately periodic, then
$$
\lim_{n \rightarrow \infty} 
{{\Pal (n) \log (\Pal (n) + 1)} \over {\rho (n)}} = 0.
$$  
Furthermore, we prove that the numerator in this result
is essentially optimal.
\end{abstract}

\section{Introduction}
Let ${\bf x} = (a_i)_{i \geq 0}$ be an infinite word over a finite alphabet
$\Sigma$.  The {\it factor complexity\/} (aka {\it subword complexity}) of
$\bf x$ is defined to be the function $\rho_{\bf x} (n)$
mapping $n$ to the number of distinct length-$n$ blocks
appearing in $\bf x$.  The {\it palindrome complexity\/} $\Pal_{\bf x} (n)$
analogously counts the number of distinct length-$n$ palindromes
appearing in $\bf x$.  

Both factor complexity and palindrome complexity have been extensively
studied by researchers in combinatorics on words.  
For factor
complexity, see, for example, \cite{Morse&Hedlund:1938} and the
surveys \cite{Allouche:1994a,Cassaigne&Nicolas:2010}.  Many classes, such as the
Sturmian words, are known to have linear factor complexity
\cite{Morse&Hedlund:1940}.   In the other direction,
Grillenberger \cite{Grillenberger:1973}
constructed words with large factor complexity.
For palindrome complexity, see, for example,
\cite{Droubay&Justin&Pirillo:2001,Allouche&Baake&Cassaigne&Damanik:2003}.

Let us look at two examples.
\begin{example}
Let ${\bf t} = 01101001\cdots$ be the famous Thue-Morse sequence,
the fixed point of the morphism $0 \rightarrow 01$, $1 \rightarrow 10$.
Then it is known that 
$$\Pal_{\bf t} (n) = \begin{cases}
2, & \text{if $n\in \{1,2,3\}$, or $n$ even and $3 \cdot 4^k < n \leq 4^{k+1}$ for $k \geq 0$;}\\
4, & \text{if $n$ even, $4^k < n \leq 3 \cdot 4^k$ for $k \geq 1$;} \\
0, & \text{otherwise.}
\end{cases}
$$
See \cite[Eq.~(4)]{Blondin-Masse&Brlek&Garon&Labbe:2008}.

It is also known that
$$ 3 = \liminf_{n \rightarrow \infty} \rho_{\bf t} (n)/n <
\limsup_{n \rightarrow \infty} \rho_{\bf t} (n)/n = 10/3.$$
See \cite{Brlek:1989,Luca&Varricchio:1989a,Avgustinovich:1994}.
Therefore $\lim_{n \rightarrow \infty} \Pal_{\bf t} (n)/\rho_{\bf t}(n) = 0$.
\end{example}

\begin{example}
Let ${\bf b} = 11011100101110111\cdots$ be 
the word consisting of the concatenation of the base-$2$ representations
of $n = 1,2,3, \ldots$.  Then clearly
$\rho_{\bf b} (n) = 2^n$, while
$\Pal_{\bf b} (n) = \Theta(2^{n/2})$.
Therefore $\lim_{n \rightarrow \infty} \Pal_{\bf b} (n)/\rho_{\bf b}(n) = 0$.
\end{example}

Examining these two very different examples, at extremes of
complexity, suggests the following:
\begin{conjecture}
If $\bf x$ is not ultimately periodic, then
$$\lim_{n \rightarrow \infty} {{\Pal_{\bf x} (n)} \over {\rho_{\bf x} (n)}} =
0.$$
\end{conjecture}
In June 2024, Jean-Paul Allouche asked me if this was true.
In fact, even more is true.  We will prove
\begin{theorem}
If $\bf x$ is not ultimately periodic, then
$$\lim_{n \rightarrow \infty} {{\Pal_{\bf x} (n) \log (\Pal_{\bf x} (n) + 1)} \over {\rho_{\bf x} (n)}} =
0.$$
\label{main}
\end{theorem}
After discussing notation and presenting
some preliminary lemmas, the proof is given in
Section~\ref{mainsec}.
Furthermore, it turns
out that the numerator $\Pal_{\bf x} (n) \log (\Pal_{\bf x} (n) + 1)$
cannot be significantly improved.  This
is addressed in Section~\ref{optsec}.

\section{Notation}

Throughout this paper,
we assume words are defined over a fixed finite alphabet
$\Sigma$, with $k = |\Sigma|$.  Infinite words are written
in boldface.  If $w$ is a finite word with $w = a_0 \cdots a_{n-1}$, then 
by $w[i..j]$ we mean $a_i \cdots a_j$, and similarly for infinite words.

All words are indexed starting at position $0$.  If
$w = xyz$ for finite (possibly empty) words $x,y,z,$ then
$x$ is a {\it prefix\/} of $w$,
$y$ is a {\it factor\/} of $w$,
and $z$ is a {\it suffix\/} of $w$.
A finite word $x = x[0..n-1]$ has {\it period\/} $p$, $1 \leq p \leq n$,
if $x[i]=x[i+p]$ for $0 \leq i < n-p$.

The set of all length-$n$ factors of an infinite word $\bf x$
is denoted by $\Fac_{\bf x} (n)$, and
the factor complexity function 
$\rho_{\bf x} (n)$ is the cardinality of $\Fac_{\bf x} (n)$.

The reversal of a finite word $w$ is written $w^R$.
A word is a {\it palindrome\/} if $w = w^R$.
The palindrome complexity function
$\Pal_{\bf x} (n)$ counts the number of distinct factors of
$\bf x$ of length $n$ that are palindromes.

In the case of both $\Palx$ and $\rhox$, if 
$\bf x$ is clear from the context, we omit it as a subscript.

\section{Previous work}

There are some previous results that are relevant and put our result in
context.

\begin{theorem}
If $\bf x$ is not ultimately periodic, then
$\rho_{\bf x} (n) < \rho_{\bf x} (n+1)$ for all $n \geq 0$.
In particular, $\rho_{\bf x} (n)$ is unbounded.
\label{thm1}
\end{theorem}
\begin{proof}
A classic result of Morse and Hedlund \cite{Morse&Hedlund:1938}.
\end{proof}

\begin{proposition}
We have $\rho_{\bf x} (m+n) \leq \rho_{\bf x} (m) \rho_{\bf x} (n)$
for $m, n \geq 0$.
\label{suba}
\end{proposition}

\begin{proof}
Every factor of length $m+n$ is a concatenation of a prefix of length $m$
and a suffix of length $n$.
\end{proof}

\begin{corollary}
The limit $\lim_{n \rightarrow \infty} {{\log \rho_{\bf x} (n)} \over n} $
exists and equals $\inf_{n \geq 1} (\log \rho_{\bf x} (n))/n$.
\end{corollary}

\begin{proof}
Follows from Fekete's lemma \cite{Fekete:1923} applied to the subadditive
sequence $(\log \rho_{\bf x} (n))_{n \geq 0}$.
\end{proof}

The next two results are {\it not\/} used in this paper, but indicate previous
related work.

\begin{theorem}[Theorem 12 of \cite{Allouche&Baake&Cassaigne&Damanik:2003}]
If $\bf x$ is not ultimately periodic, then
$$\Pal_{\bf x} (n) < {16 \over n} \rho_{\bf x} (n + \lfloor n/4 \rfloor)$$
for all $n \geq 1$.
\label{eight}
\end{theorem}

\begin{theorem}[Theorem 1.2 of \cite{Balazi&Masakova&Pelantova:2007}]
Suppose $\bf x$ is uniformly recurrent.
\begin{itemize}
\item[(a)] If the set of factors of $\bf x$ is not closed under
reversal, then $\Pal_{\bf x} (n) =0$ for all sufficiently large
$n$.
\item[(b)] If the set of factors of $\bf x$ is closed under reversal,
then
$$\Pal_{\bf x} (n) + \Pal_{\bf x} (n+1) \leq 
\rho_{\bf x}(n+1) - \rho_{\bf x}(n) + 2$$
for $n \geq 0$.
\end{itemize}
\label{nine}
\end{theorem}

However, neither Theorem~\ref{eight} nor Theorem~\ref{nine} seem strong enough to
prove Theorem~\ref{main}.

\section{Lemmas}

In this section and the next, the goal is to prove Theorem~\ref{main}.
We fix an infinite word $\bf x$ and
omit it as a subscript on $\Pal$ and $\rho$.
The idea of the proof is as follows.  
The case where subword complexity is exponentially large
is easy.  Otherwise, we show that for each $D \geq 1$
and all sufficiently large $n$ there is a way of associating,
with each length-$n$ palindrome $u$ appearing in $\bf x$,
a list $S_u$ of length-$n$ factors of $\bf x$ of cardinality
$D+1$ such that the sets $S_u$ are `almost' pairwise disjoint---hence
giving `roughly' $(D+1)\cdot \Pal(n) $ factors of length $n$.
Thus for all these $n$ the ratio of ordinary
factors to palindrome factors is roughly $D+1$.

We will need two lemmas.

\begin{lemma}
Let $w$ be a palindrome, and $x$ a nonempty proper palindromic prefix of $w$.
Then $|w|-|x|$ is a period of $w$.
\label{lem1}
\end{lemma}

\begin{proof}
This is a standard lemma; see \cite[Lemma 2]{Frid&Puzynina&Zamboni:2013}, for
example.  It has a $1$-line proof:  if $x$ is a prefix of $w$ then 
$x^R = x$ is a suffix of $w$, and so $w$ has period $|w|-|x|$.
\end{proof}

Now we define an ``exceptional set'' ${\cal E}_D (n)$ and bound
its cardinality.
\begin{lemma}
Let $D \geq 1$ be an integer and $n > 2D$. 
Define
$$ {\cal E}_D (n) = \{ v \in \Fac (n) \suchthat
	v \text{ has a prefix $z$ with $|z| \geq n-2D$
	and period $\leq 2D$} \},$$
and set $E_D(n) := |{\cal E}_D (n)|$.
Then 
\begin{equation}
(D+1) \Pal (n) \leq \rho (n) + D E_D (n)
\label{eq1}
\end{equation}
and
\begin{equation}
E_D(n) \leq 2D(2D+1) \rho(2D)^2.
\label{eq2}
\end{equation}
\label{exceptional}
\end{lemma}

\begin{proof}
For each 
length-$n$ palindrome $u$ occurring in $\bf x$, let $i_u$ be the starting
position of the first occurrence of $u$ in $\bf x$,
and define
$$T(u,d) = {\bf x}[i_u+d..i_u+d+n-1]$$
for $0 \leq d \leq D$.  So $T(u,d)$ is a length-$n$ factor of $\bf x$,
beginning at position $i_u + d$.

\begin{claim}
For each fixed $d$, $0 \leq d \leq D$, the map
$u \rightarrow T(u,d)$ is injective.
\label{claim1}
\end{claim}

\begin{proof}
First fix $d$.
We have $u = {\bf x}[i_u..i_u+n-1]$
and $T(u,d) = {\bf x}[i_u+d..i_u+d+n-1]$.
The length $n-d$ prefix of $T(u,d)$ is
${\bf x}[i_u+d..i_u+n-1]$, which is the
length $n-d$ suffix of $u$.  Assume $T(u,d) = T(v,d)$.
This implies that the length $n-d$ suffix of $u$ and $v$ coincide.
Since $d \leq D < n/2$,
we have $n-d > n/2$.  But every palindrome of length $n$
is uniquely determined by a suffix of length $> n/2$.  So $u = v$.
Thus the map is injective.
\end{proof}

\begin{claim}
Suppose $z = T(u,d) = T(v,e)$ for length-$n$ palindromic factors
$u,v$ of $\bf x$, and $0 \leq d < e \leq D$. Then
$z \in {\cal E}_D(n)$.
\label{claim2}
\end{claim}

\begin{proof}
We show that $z \in {\cal E}_D (n)$.
Write $u = u_0 \cdots u_{n-1}$, so that
$z = u_d \cdots u_{d+n-1}$.
The prefix of $z$ of length $n-2d$ is then
$z' = u_d \cdots u_{n-d-1}$.
Since $u = u_0 \cdots u_{d-1} z' u_{n-d} \cdots u_{n-1}$,
the word $z'$ results from removing the prefix and suffix of length
$d$ from the palindrome $u$.  Thus $z'$ is a palindrome.
So $u_{d+r} = u_{n-d-1-r}$ for $0 \leq r < n-2d$.
So $z$ has a palindromic prefix $w$ of length $n-2d$.

By exactly the same argument, $z = T(v,e)$ has a palindromic
prefix $y$ of length $n-2e$.

Since $d<e$, the word $y$ is a nonempty proper
palindromic prefix of $w$. By Lemma~\ref{lem1}, the word $w$ has period
\[
   |w|-|y|=(n-2d)-(n-2e)=2(e-d)\le 2D.
\]
Moreover, $|w|=n-2d\ge n-2D$. Hence $z$ has a prefix of length at least
$n-2D$ and period at most $2D$, so $z$ lies in the
exceptional set ${\cal E}_D(n)$.
\end{proof}

Now we are ready to finish the proof of Lemma~\ref{exceptional}.  For
$z \in \Fac (n)$, define
$m(z)$ to be the number of pairs $(u,d)$ with
$u$ a length-$n$ palindromic factor of $\bf x$,
$0 \leq d \leq D$, and $T(u,d) = z$.

By Claim~\ref{claim1}, for each $d$ there is at most one
palindrome $u$ with $T(u,d) = z$.  So $m(z) \leq D+1$
for each $z \in \Fac (n)$.

If $z \not\in {\cal E}_{D} (n)$, then $m(z) \leq 1$, because
Claim~\ref{claim1} rules out $T(u,d) = T(v,d) = z$ for
distinct $u,v$, while $T(u,d) = T(v,e) = z$ for
distinct $d,e$ forces $z \in {\cal E}_{D}(n)$ by Claim~\ref{claim2}.

Hence we have
\begin{align*}
(D+1) \Pal (n) &= \sum_{z \in \Fac (n)} m(z) \\
&\leq |\Fac (n) \setminus {\cal E}_D (n)| + 
(D+1) |\Fac (n) \cap {\cal E}_D (n) | \\
&= |\Fac (n)| + D|\Fac(n) \cap {\cal E}_D (n)|  \\
&= \rho(n) + D|\Fac(n) \cap {\cal E}_D (n) | \\
&\leq \rho(n) + DE_{D} (n).
\end{align*}

It remains to bound $E_D(n)$.   Suppose $v \in {\cal E}_D (n)$.
Then $v$ has a prefix $z$ of length $N \geq n-2D$ and
period $p \leq 2D$.  Thus $v$ is completely determined by
specifying
\begin{itemize}
\item the period $p$, $1 \leq p \leq 2D$, giving $2D$ choices;
\item the first $p$ letters of $v$, which determine $z$ and provide $\rho(2D)$
choices;
\item the length $r = n-N$ of the remaining suffix of $v$,
where $0 \leq r \leq 2D$,
giving
at most $2D+1$ choices;
\item the suffix of $v$ of length $r$, giving at most $\rho(2D)$ choices.
\end{itemize}
It follows that $E_D (n) \leq 2D(2D+1) \rho(2D)^2$.
\end{proof}

\section{Proof of the main result}
\label{mainsec}

\begin{proof}[Proof of Theorem~\ref{main}.]
Assume that $\bf x$ is not ultimately periodic.  The goal is
to prove that the ratio
$$R(n) :=  {{\Pal(n) \log (\Pal (n) + 1)} \over {\rho (n)}} $$
tends to $0$ as $n \rightarrow \infty$.
Define 
$$h = \lim_{n \rightarrow \infty} {{\log \rho (n)} \over n} ,$$
which exists by Proposition~\ref{suba}.

The easy case is $h > 0$.  A length-$n$ palindrome is completely
determined by its prefix of length $\lceil n/2 \rceil$, and
such a prefix is also a factor of $\bf x$.
So $\Pal (n) \leq \rho (\lceil n/2 \rceil)$.
Since $\log \rho(n) = hn + o(n)$, we have
$\log \rho (\lceil n/2 \rceil) = (h/2) n + o(n)$.
If $\Pal (n) = 0$, then $R(n) = 0$.
Otherwise, for those $n$ where $R(n) > 0$, we have
\begin{align*}
\log R(n) &\leq
\log \rho(\lceil n/2 \rceil) + \log \log(\rho(\lceil n/2 \rceil ) + 1)
 - \log \rho (n)  \\
 & = -{h \over 2} n + o(n),
\end{align*}
which tends to $-\infty$.  Hence $R(n)$ tends to $0$.

The harder case is $h = 0$.  
Suppose, to get a contradiction, that $R(n)$ does not tend to $0$.
Then there exist \(\varepsilon>0\) and an
increasing sequence $(n_j)_{j \geq 1}$ such that
$ R(n_j)\ge \varepsilon$ for all $j\geq 1$.
Along this particular sequence, \(\Pal(n_j)\) must tend to infinity.  Indeed, if some infinite subsequence \(n_{j_\ell}\) satisfied \(\Pal(n_{j_\ell})\le C\), then
\[
  \Pal(n_{j_\ell})\log(\Pal(n_{j_\ell})+1)
  \le C\log(C+1),
\]
while \(\rho(n_{j_\ell})\to\infty\), since \(\bf x\) is not ultimately periodic.  Hence
$  R(n_{j_\ell})\to 0$,
contradicting $R(n_j)\ge \varepsilon$.  Thus
$ \Pal(n_j)\to\infty$.  

Fix \(A>0\), and set
\[
  D_j=\left\lfloor A\log(\Pal(n_j)+1)\right\rfloor.
\]
Since \(\Pal(n_j)\to\infty\), we have \(D_j\to\infty\).  Also, because each length-\(n\) palindrome is determined by its first \(\lceil n/2\rceil\) letters,
we have $ \Pal(n)\le \rho(\lceil n/2\rceil)$.
Using \(\log\rho(m)=o(m)\), this gives
$ \log(\Pal(n_j)+1)=o(n_j)$.
Consequently $D_j=o(n_j)$,
and in particular \(D_j<n_j/2\) for all sufficiently large \(j\).  

We now apply Lemma~\ref{exceptional}.
Since \(\log\rho(m)=o(m)\), Eq.~\eqref{eq2} implies that
\[
  \log (E_D(n) + 1)=o(D)
  \qquad(D\to\infty),
\]
uniformly in \(n\).  Substituting \(D=D_j\), and using
\[
  D_j\sim A\log(\Pal(n_j)+1),
\]
we get
\[
  \log E_{D_j}(n_j)
  =o(\log(\Pal(n_j)+1)).
\]
Thus
\[
  E_{D_j}(n_j)=\Pal(n_j)^{o(1)}=o(\Pal(n_j)),
\]
because \(\Pal(n_j)\to\infty\).  Applying Eq.~\eqref{eq1} with
\(n=n_j\) and \(D=D_j\), we obtain
\[
  \rho(n_j)
  \ge (D_j+1)\Pal(n_j)-D_jE_{D_j}(n_j)
  = (1-o(1))D_j\Pal(n_j).
\]
Since \(D_j\sim A\log(\Pal(n_j)+1)\), this gives
\[
  \rho(n_j)
  \ge
  (1-o(1))A \Pal (n_j)\log(\Pal (n_j)+1).
\]
Equivalently,
\[
  \limsup_{j\to\infty}
  \frac{\Pal(n_j)\log(\Pal(n_j)+1)}{\rho(n_j)}
  \le \frac{1}{A}.
\]
Because \(A>0\) was arbitrary, the right-hand side can be made smaller than \(\varepsilon\), contradicting \(R(n_j)\ge \varepsilon\).  Therefore \(R(n)\to 0\), as claimed.
\end{proof}

\section{Optimality}
\label{optsec}

It is natural to wonder if the $\Pal(n) \log(\Pal(n))$ in the numerator
of the quantity in
Theorem~\ref{main} is optimal.  We show that it (essentially) is.

\begin{theorem}
Let $f:\Enn\to[0,\infty)$ be a nondecreasing function.  The following are
equivalent.
\begin{itemize}
\item[(i)] For every non-ultimately periodic infinite word $\bf x$ over a finite
      alphabet,
\[
   \frac{f(\Pal(n))}{\rho(n)}\to 0.
\]
\item[(ii)] $f(t)=O(t\log(t+1))$.
\end{itemize}
Moreover, if $f(t)/(t\log(t+1))$ is unbounded, then there exists a
non-ultimately periodic infinite word $\bf x$ for which
\[
   \limsup_{n\to\infty}
   \frac{f(\Pal(n))}{\rho(n)}=+\infty.
\]
\label{optimal}
\end{theorem}

\begin{proof}
The implication (ii) $\Rightarrow$ (i) follows immediately from
Theorem~\ref{main}.  
Indeed, by Theorem~\ref{main},
\[
   \frac{\Palx(n)\log(\Palx(n)+1)}{\rhox(n)}\to 0.
\]
Thus the assertion follows for those $n$ with large $\Palx(n)$; for the
remaining $n$, the numerator is bounded while $\rhox(n)\to\infty$.

We now prove the converse (i) $\Rightarrow$ (ii) in a strong form, by
constructing a word $\bf x$ with many palindromes.
Assume that
\[
   \frac{f(t)}{t\log(t+1)}
\]
is unbounded.  We first record a simple consequence.  Since $f$ is
nondecreasing, the sequence
$   \frac{f(2^r)}{2^r r}$
is also unbounded.  To see this, choose an increasing sequence
integers $t_j$ for which
$f(t_j)/(t_j\log(t_j+1))\to\infty$, and put
$r_j=\lceil \log_2 t_j\rceil$.  Then $t_j\leq 2^{r_j}\leq 2t_j$, and
$$r_j=O(\log(t_j+1)).$$
Therefore
\[
   \frac{f(2^{r_j})}{2^{r_j}r_j}
   \geq
   c\,\frac{f(t_j)}{t_j\log(t_j+1)}
\]
for sufficiently large $j$ and some absolute constant $c>0$,
and hence the right-hand side is unbounded.

We now construct a word $\bf x$ over the alphabet
   $\Sigma=\{0,1,2,3\}$
for which
$$   \limsup_{n\to\infty}
   \frac{f(\Palx(n))}{\rho_{\bf x}(n)}=+\infty.$$

For integers $s,r\geq 1$ and a word
$\varepsilon=\varepsilon_0 \cdots \varepsilon_{r-1}\in\{1,2\}^r$, define
\[
\begin{aligned}
   p(s,r,\varepsilon)
   ={}&0^s\varepsilon_0 0^s\varepsilon_1\cdots 0^s\varepsilon_{r-1} 0^s
      \varepsilon_{r-1} 0^s\varepsilon_{r-2}\cdots 0^s\varepsilon_0 0^s.
\end{aligned}
\]
This word is a palindrome.  It has $2r$ nonzero letters and $2r+1$ blocks
of the form $0^s$, so its total length is
$ n(s,r)=(2r+1)s+2r$.
For fixed $s$ and $r$, the $2^r$ words $p(s,r,\varepsilon)$ are
clearly all distinct as
$\varepsilon$ ranges over $\{1,2\}^r$.

We choose parameters recursively.  Let $s_1=1$ and $S_0=0$.  Suppose that
$s_k$ and $S_{k-1}$ have already been chosen.  Since
$f(2^r)/(2^r r)$ is unbounded, we may choose $r_k$ so large that, with
$M_k=2^{r_k}$, both inequalities
\begin{equation}
   M_k\geq S_{k-1}+1
\label{ineq1}
\end{equation}
and
\begin{equation}
   \frac{f(M_k)}{M_k r_k}
   \geq 21k(s_k+1)
\label{ineq2}
\end{equation}
hold.  Set
\[
   n_k=n(s_k,r_k)=(2r_k+1)s_k+2r_k.
\]
Let $B_k$ be the concatenation, in any order, of all the $M_k$ palindromes
\[
   p(s_k,r_k,\varepsilon),\qquad \varepsilon\in\{1,2\}^{r_k},
\]
each followed by the delimiter $3$.  Thus
$ |B_k|=M_k(n_k+1)$.
Define
\[
   S_k=S_{k-1}+|B_k|,
   \qquad
   s_{k+1}=n_k+1, 
\]
and set
\[
   {\bf x} =B_1B_2B_3\cdots.
\]
Note that $|B_1 B_2 \cdots B_k| = S_k$.

The word $\bf x$ is not ultimately periodic.  Indeed, $s_k\to\infty$, so
$\bf x$ contains arbitrarily long blocks of $0$'s.  If $\bf x$ were ultimately
periodic with period $T$ from some point onward, then a sufficiently late
block $0^T$ would force the periodic tail from that block onward to be
identically $0$.  But after every such block there are still later
occurrences of the symbols $1$, $2$, and $3$, a contradiction.

Now fix $k$ and examine factors of length $n_k$.  The block $B_k$ contains
all $M_k$ palindromes $p(s_k,r_k,\varepsilon)$ of length $n_k$.  Hence
\[
   \Palx(n_k)\geq M_k.
\]

We next bound $\rho_{\bf x}(n_k)$.  A length-$n_k$ factor starts either
before $B_k$, inside $B_k$, or after $B_k$.
There are exactly $S_{k-1}$ starting positions before $B_k$, so these
contribute at most $S_{k-1}$ distinct factors.  There are $|B_k|$ starting
positions inside $B_k$, so these contribute at most
$ |B_k|=M_k(n_k+1)$ distinct factors.
It remains to count factors starting after $B_k$.  For every $m>k$ we have
\[
   s_m\geq s_{k+1}=n_k+1.
\]
Thus, in the tail $B_{k+1}B_{k+2}\cdots$, any two nonzero symbols are
separated by more than $n_k$ positions: between consecutive nonzero
symbols there is a block of at least $n_k+1$ zeros.  Therefore a factor of
length $n_k$ starting after $B_k$ contains at most one nonzero symbol.
Such a factor is either
$   0^{n_k}$ or has the form
\[
   0^j c\,0^{n_k-j-1},
   \qquad 0\leq j<n_k,
   \qquad c\in\{1,2,3\}.
\]
There are at most $3n_k + 1$ such words.  Consequently,
\[
   \rho_{\bf x}(n_k)
   \leq S_{k-1}+M_k(n_k+1)+3n_k + 1.
\]
By \eqref{ineq1} we have $S_{k-1}\leq M_k$.  Since $M_k,n_k\geq 1$, we get
\[
   S_{k-1}+M_k(n_k+1)+3n_k+1
   \leq 7M_kn_k.
\]
Also we have
\[
   n_k=(2r_k+1)s_k+2r_k\leq 3r_k(s_k+1),
\]
and thus
\[
   \rho_{\bf x}(n_k)
   \leq 21M_kr_k(s_k+1).
\]
Using the fact that $f$ is nondecreasing, \eqref{ineq2},
and the lower bound
$\Palx(n_k)\geq M_k$, we obtain
\begin{displaymath}
   \frac{f(\Palx(n_k))}{\rho_{\bf x}(n_k)}
   \geq
   \frac{f(M_k)}{21M_kr_k(s_k+1)} \geq k.
\end{displaymath}
Therefore
\[
   \limsup_{n\to\infty}
   \frac{f(\Palx(n))}{\rho_{\bf x}(n)}=+\infty.
\]
This completes the proof.
\end{proof}

\begin{corollary}
For every unbounded nondecreasing function $g$, there is a
non-ultimately periodic infinite word $\bf x$ such that
\[
   \limsup_{n\to\infty}
   \frac{\Palx(n)\log(\Palx(n)+1)
          g(\Palx(n))}{\rho_{\bf x}(n)}=+\infty.
\]
\end{corollary}

\begin{proof}
Apply Theorem~\ref{optimal} to
\[
   f(t)=t\log(t+1)g(t).
\]
\end{proof}

In 2001, according to \cite[Remark 9, p.~27]{Allouche&Baake&Cassaigne&Damanik:2003}, Christian Choffrut asked if it is always the case that
$\Palx (n) = O(\sqrt{\rhox(n)})$.  Theorem~\ref{optimal} shows that the answer 
is no.

\section{Acknowledgments}

Much of this paper was written by ChatGPT 5.5 Pro.  All arguments were
checked and rewritten by the author, who takes full responsibility
for correctness.  We thank Jean-Paul Allouche and Edita Pelantov\'a
for their comments.

\end{document}